\documentclass[final, 12pt,reqno]{amsart}
\usepackage{amssymb,amsmath,graphicx,amsfonts,euscript}
\usepackage{color}
\usepackage[pagewise]{lineno}
\usepackage{showkeys}
\usepackage{fancyhdr}
\usepackage[margin=1in]{geometry}
\usepackage{titletoc}
\usepackage{appendix}
\usepackage{color}
\usepackage{hyperref}

\newtheorem{theorem}{\bf Theorem}[section]

\newtheorem{proposition}[theorem]{\bf Proposition}

\newtheorem{remark}{\bf Remark}[section]

\newtheorem{lemma}[theorem]{\bf Lemma}

\allowdisplaybreaks[4]

\newcommand{\beq}{\begin{equation}}
	\newcommand{\eeq}{\end{equation}}
\newcommand{\ben}{\begin{eqnarray}}
	\newcommand{\een}{\end{eqnarray}}
\newcommand{\beno}{\begin{eqnarray*}}
	\newcommand{\eeno}{\end{eqnarray*}}

\title[Transition threshold for 2-D Couette flow]{Transition threshold for the 2-D Couette flow in whole space via Green's function}

\author[Gaofeng Wang and  Weike Wang ]{\sc Gaofeng Wang$^{1}$, Weike Wang$^{2}$}

\address{$^1$ School of Mathematical Sciences , Shanghai Jiao Tong University, Shanghai, 200240,  P.R.China.}
\email{1462047796@qq.com}

\address{$^2$ School of Mathematical Sciences, CMA-Shanghai and Institute of Natural Science, Shanghai Jiao Tong University, Shanghai, 200240, P.R.China. }
\email{wkwang@sjtu.edu.cn}

\begin{document}
	\bibliographystyle{abbrv}
	\maketitle
	\vskip .01in
	\begin{center}
		\sc Abstract
	\end{center}
	
	\vskip .05in
	
	In this paper, we investigate the transition threshold problem concerning the 2-D Navier-Stokes equations in the context of Couette flow $(y,0)$ at high Reynolds number $Re$ in whole space. By utilizing Green's function estimates for the linearized equations around Couette flow, we initially establish refined dissipation estimates for the linearized Navier-Stokes equations with a precise decay rate $(1+t)^{-1}.$ As an application, 
we prove that if the initial perturbation of vorticity satisfies$$\|\omega_{0}\|_{H^{1}\cap L^1}\leq c_0\nu^{\frac{3}{4}}$$ for some small constant $c_0$ 
independent of the viscosity $\nu$, then we can reach the conclusion that the solution remains within $O\left( \nu ^{\frac{3}{4}}\right) $ of  the Couette  flow.

	\maketitle
	
	\titlecontents{section}[0pt]{\vspace{0\baselineskip}\bfseries}
	{\thecontentslabel\quad}{}%
	{\hspace{0em}\titlerule*[10pt]{$\cdot$}\contentspage}
	
	\titlecontents{subsection}[1em]{\vspace{0\baselineskip}}
	{\thecontentslabel\quad}{}%
	{\hspace{0em}\titlerule*[10pt]{$\cdot$}\contentspage}
	
	\titlecontents{subsubsection}[2em]{\vspace{0\baselineskip}}
	{\thecontentslabel\quad}{}%
	{\hspace{0em}\titlerule*[10pt]{$\cdot$}\contentspage}


	\vskip .3in

	
	\section{Introduction}
	In this paper, we first consider the stability of Couette flow under the two-dimensional incompressible Navier-Stokes  equations in whole space $(x,y)\in \mathbb{R}^2$
	\begin{equation}\label{eq:1}
		\left\{
		\begin{aligned}
			&\partial _tv+v\cdot\nabla v-\nu \Delta v+\nabla q=0,\\
			&\nabla\cdot v=0,\\
			&v(0,x,y)=v_{0}(x,y),
		\end{aligned}
		\right.
	\end{equation}
where $\nu$  is the viscosity coefficient. Here, the unknowed  $v(t,x,y)=(v^1(t,x,y),v^2(t,x,y)),$ $q(t,x,y)$ denote the 2-D velocity field, pressure, respectively.

The  goal of this paper is to understand the
stability and large-time behavior of perturbations near the Couette  flow.  The Couette flow 
 $U_s=(y,0)$
is the  steady state solution of the \eqref{eq:1}.
Now, let's introduce the perturbation function
 $v=u+(y,0),$ then this converts the system of  \eqref{eq:1} into 
perturbation equations  which satisfies
 \begin{equation}\label{eq:2}
 	\left\{
 	\begin{aligned}
 		&\partial _tu+y\partial_x u+	\begin{pmatrix}
 			u^2\\
 		0
 		\end{pmatrix}-\nu \Delta u+u\cdot\nabla u+\nabla q=0,\\
 		&\nabla\cdot u=0,\\
 		&u(0,x,y)=u_{0}(x,y).\\
 	\end{aligned}
 	\right.
 \end{equation}
Due to the existence of the pressure term in system $\eqref{eq:2}$, there are obstacles to deal with the velocity field directly. 
By takng the curl  on both sides of $\eqref{eq:2}$, not only can the pressure term be eliminated, but also the vector equation can be transformed into a scalar equation.
 The vorticity $\omega$ and the
stream function $\varphi$	 are defined by
$$\omega=\nabla\times u=\partial_y u^1-\partial_x u^2,\quad u=\nabla^{\perp}\varphi=(\partial_y\varphi,-\partial_x \varphi),\quad \Delta \varphi=\omega,$$ respectively, which satisfies
\begin{equation}\label{eq:3}
	\left\{
	\begin{aligned}
		&\partial _t\omega+y\partial_x\omega -\nu \Delta \omega=-\nabla\cdot (u \omega),\\
		&u=\nabla^{\perp}\Delta^{-1}\omega,\\
	&\omega(0,x,y)=\omega_{0}(x,y).\\
		\end{aligned}
	\right.
\end{equation}
  The hydrodynamic stability at high Reynolds numbers has been a significant research topic ever since Reynold's famous experiment 
  (see \cite{O. Reynolds.1883}), in which fluid was pumped into a pipe under various conditions and he demonstrated that laminar flow becomes spontaneously unstable at sufficiently high Reynolds numbers. This experiment primarily focused on understanding the instability of laminar flow and its transition to turbulence. Numerous studies have been conducted to investigate how laminar flows become unstable and transition into turbulence (see \cite{ S.Orszag.1980,Schmid.2001,Yaglom.2012}). However, there is currently no evidence of any linear instability in three-dimensional pipe flows at finite Reynolds numbers (see \cite{Drazin.1982}). It should be noted that although nonlinear instability has been observed in Reynold's experiment (see \cite{O. Reynolds.1883}), and it is even more concerning that for laminar flows with unstable eigenvalues predicted by linear theory at high Reynolds numbers, experiments and computer simulations typically show different instabilities compared to those predicted by linear theory at lower Reynolds numbers (see \cite{Schmid.2001}). These phenomena are commonly referred to as subcritical or bypass transitions in fluid mechanics.
  
The understanding of these subcritical transitions has been the subject of numerous investigations (see \cite{Chapman.2002} and references therein). One proposed explanation, dating back to Kelvin (see \cite{Kelvin.1887}) posits that the basin of attraction for laminar flow shrinks as (see $\nu\rightarrow 0$), rendering the flow nonlinearly unstable under small yet finite perturbations. Consequently, a pivotal inquiry initially posed by Trefethen et al. (see \cite{Trefethen.1993}) is to examine the transition threshold.
  The following mathematical version was formulated  by Bedrossian, Germain, and Masmoudi \cite{Bedrossian.2019}:
  
Given a norm $||\cdot||_X$, determine a $\gamma=\gamma(X)$ such that 
$$||u_0||_X<C \nu^{\gamma}\rightrightarrows stability,$$
$$||u_0||_X\gg C \nu^{\gamma}\rightrightarrows instability,$$
$\gamma$: the transition threshold in the applied literature.
The result of the threshold is intricately linked to topology, as an evident from the aforementioned definition of transition threshold.

After reviewing previous studies on the transition threshold of steady states, we would like to present our main result. There have been extensive researches in applied mathematics and physics (see \cite{Duguet.2010, Lundbladh.1994,S.Orszag.1980,Reddy.1998,Yaglom.2012}) dedicated to calculating transition threshold for various flows such as Couette flows and Poiseuille flows. Regarding Couette flow stability threshold problem, several works have been conducted (see \cite{Bedrossian.2020,Bedrossian.2017,Bedrossian.2016,Bedrossian.2018}).
When $\Omega=\mathbb{T}\times\mathbb{R}\times\mathbb{T},$ Bedrossian, Germain and Masmoudi (see \cite{Bedrossian.2020,Bedrossian.2022,Bedrossian.2017}) proved the transition threshold $\gamma= 1$ and $\gamma\leq\frac{3}{2}$  using the Fourier multiplier method if the perturbation for velocity belongs to the Gevrey class and Sobolev class respectively. This result has been further improved $\gamma\leq 1$ by Wei and Zhang (see \cite{Wei.2021} ) through resolvent estimates when the perturbation for velocity is in the Sobolev class $H^{2}$. The same result $\gamma\leq 1$ also holds in a finite channel with non-slip boundary conditions when the perturbation for velocity is in the Sobolev class $H^{2}$ when $\Omega=\mathbb{T}\times[-1,1]\times\mathbb{T}$ (see \cite{Chen.2020}). From this result, it can be inferred that non-slip boundary conditions contribute to the development of a boundary layer effect, but they are not the primary factor driving strong instability.
When $\Omega=\mathbb{T}\times \mathbb{R},$ the same transition threshold for the Couette is obtained $\gamma\leq\frac{1}{2}$ (see \cite{Bedrossian.2018,Masmoudi.2020}) if the perturbation for velocity is in Sobolev space  $H^2$ and in $H_x^{log}L_y^2$, respectively, which has been improved to $\gamma\leq\frac{1}{3}$ if the perturbation for vorticity is in Sobolev space   Sobolev space  $H^{\sigma}(\sigma\geq 40)$ (see \cite{Masmoudi.2022}) via the time-dependent multiplier method. The
	result $\gamma\leq \frac{1}{2}$ also holds in a finite channel with non-slip boundary condition  if the perturbation  for velocity is  in Sobolev class $H^{2}$ when $\Omega=\mathbb{T}\times[-1,1]$ (see \cite{L.2020}). This result implies that the boundary layer effect induced by non-slip
	boundary condition is not the main effect resulting in the strong instability. For the perturbation for vorticity in the Gevery class (see \cite{Bedrossian.2016,Li.2022}).
	
 Moreover, for 2D plane Poiseuille flow, when $\Omega=\mathbb{T}\times \mathbb{R},$ Coti Zelati et al. (see  \cite{Zelati.2020}) proved the transition threshold $\gamma\leq \frac{3}{4}+2\epsilon $  (for any $\epsilon>0$) via the hypocoercivity method,  which has been improved to $\gamma\leq \frac{2}{3}+$ by Del Zotto (see \cite{Zotto.2023}). Ding and Lin (see \cite{Ding.2022}) were to treat with the case
 in a finte channel $\mathbb{T}\times (-1,1)$ with Navier-slip boundary condition, which proved the transition threshold $\gamma\leq \frac{3}{4}$ based on the resolvent estimate and gived the enhnaced dissipation with decay rate $e^{-c\sqrt{\nu}t}.$ It is worth mentioning that they (see \cite{Ding.2024}) have recently improved their results to $\gamma\leq\frac{2}{3}$ for Navier-slip boundary condition  due to  the improvement of  the inviscid damping effect.
 
For the 3D case, Chen, Wei, and Zhang (see \cite{Chen.2023}) established the linear stability of 3D pipe Poiseuille flow through resolvent estimates. Recently, Chen, Ding et al. (see \cite{Chen.2024}) utilized the method of resolvent estimates to determine the transition threshold $\gamma\leq \frac{7}{4} $ for 3D plane Poiseuille flow in a finite channel $\mathbb{T}\times [-1,1]\times\mathbb{T}$. It is noteworthy that Wang and Xie (see \cite{Wang.2022}) demonstrated the structural stability of Hagen-Poiseuille flows in a pipe under steady conditions.

The stability of the system is attributed to two main mechanisms: dissipative enhancement and inviscid damping, with the latter being a consequence of mixing induced by the shear flow (see \cite{Bedrossian.2015,Ionescu.2019,Zhao.2020} for inviscid damping results in Euler equation). Due to diffusion effects, the decay rate of high frequency components is significantly faster than that of low frequency components. Energy transfer from low to high frequencies occurs through mixing, resulting in dissipative enhancement.
This phenomenon of enhanced dissipation  has been extensively observed and studied in physics literature (see \cite{Bernoff.1994,Latini.2001,Rhines.1983,Thomson.1887}). Recently, it has garnered significant attention from the mathematics community and substantial progress has been made. Constantin et al. were among the first to rigorously establish enhanced dissipation through diffusive mixing (see \cite{Constantin.2008}). Since then, numerous remarkable results have emerged. Notably, there has been intensive investigation into the stability of shear flows with respect to passive scale equations and Navier-Stokes equations in a series of outstanding papers  (see \cite{Bedrossian.2020,Bedrossian.2017,Bedrossian.2016,Bedrossian.2018,Chen.2020,Wei.2021}).
 
The observation reveals that prior research necessitates the periodicity of the $x$ direction and guarantees discrete frequencies. This requirement stems from the fact that one of the primary objectives of these studies is to estimate the spectrum of the operator within the linear component of the equation. The situation in this case is unequivocal, as the non-zero mode exhibits a distinct spectral gap that facilitates the capture of dissipative enhancement. The zero mode  typically possesses favorable properties contributing to stability. The enhanced dissipation effect in such a situation becomes ambiguous when considering $x\in \mathbb{R}$ without this requirement of periodicity of the $x$ direction, as the corresponding spectra are continuous and it is no longer possible to distinguish between zero and non-zero modes as previously. Nevertheless, when considering shear flow in whole space, such an approach becomes entirely invalid, demanding a completely novel and distinct idea. Consequently, addressing the stability problem associated with Cauchy problem for Navier-Stokes equations remains challenging and has been scarcely explored.

The noteworthy aspect is that Deng, Shi and Wang (see \cite{Deng.2023}) investigate the 3-D parabolic-parabolic and parabolic-elliptic Keller-Segel models with Couette flow in $\mathbb{R}^3$. They establish that the blow-up phenomenon of the solution can be suppressed by enhanced dissipation of large Couette flows. In this study, they employ Green's function method  which is different form the method of the periodic spatial variable $x$ to describe the enhanced dissipation through a more precise space-time structure and obtain global existence along with pointwise estimates of the solutions. The authors introduces a novel approach to suppressing the bursting of chemotactic model solutions by shear flow in whole space using the Green's function method. This classical method has been combined with modern micro-local analysis and harmonic analysis methods over the past two decades to study nonlinear equation solutions and obtain significant results (see \cite{Liu.2017,Liu.2010,Liu.2009,Liu.1998}). One advantage of utilizing Green's functions for discussing these partial differential equations is that it allows for a clear and localized understanding of how various mechanisms interact. The fine analysis in paper (see \cite{Deng.2023}) reveals that the Gauss kernel with Couette flow effectively suppresses solution blow up  through the utilization of a large constant. Taking inspiration from this paper, our aim is to explore innovative approaches in discussing the reinforcing effect of shear flow on the dissipative mechanism through the utilization of Green's function. This involves investigating the stability of shear flows at high Reynolds number.

The stability of Couette flow under the 2-D incompressible Navier-Stokes equations in whole space is examined in this paper. The difficulty caused by $x\in \mathbb{R}$ is overcome by employing the Green's function method, which enables us to obtain a comprehensive and precise description of the solution.
 Our main result is stated as follows:

\begin{theorem}\label{thm1}
	Suppose  that $\omega_0\in H^1(\mathbb{R}^2)\cap L^{1}(\mathbb{R}^2)$. There exists constants $\nu_0$  and $c_0, C> 0$ independent  of $\nu$  so that if 
\begin{equation}
\|\omega_{0}\|_{H^{1}\cap L^1}\leq c_0\nu^{\frac{3}{4}}
\end{equation}
	for some sufficiently small $c_0,0<\nu\leq \nu_0,$ 
then the solution of the system \eqref{eq:3} is global in time and satisfies the subsequent stability estimate 
	\begin{equation}
\|\omega\|_{L^\infty L^2}\leq Cc_0\nu^{\frac{3}{4}},
\end{equation}
as well as the decay estimate 
\begin{equation}
	\|\omega(t)\|_{ L^2}\leq Cc_0(1+t)^{-1}\nu^{\frac{3}{4}}.
\end{equation}
\end{theorem}
\begin{remark}
	To the best of our knowledge, Theorem \ref{thm1} represents the first finding on the transition threshold for Couette flow in the whole space. Asymptotic stability is guaranteed for initial perturbations that satisfy
$$\|\omega_{0}\|_{H^{1}\cap L^1}\leq c_0\nu^{\frac{3}{4}}.$$
\end{remark}
\begin{remark}
	Our main results are presented in Theorem \ref{thm1}, which demonstrate an enhanced dissipation with a decay rate of $\left( 1+t\right) ^{-1}$ and identify the corresponding transition threshold $\gamma\leq \frac{3}{4}$. In comparison to (see \cite{Bedrossian.2020,Bedrossian.2022,L.2020,Masmoudi.2022}), our objective is to address the case of the entire space, and our approach significantly differs from (see \cite{Bedrossian.2020,Bedrossian.2022,L.2020,Masmoudi.2022}). Our proof relies on the Green's function method that yields the enhanced dissipation of $\left( 1+t\right) ^{-1}$ and subsequently determines the transition threshold $\gamma\leq \frac{3}{4}$.
\end{remark}
\begin{remark}
	The stability of the Couette flow on \eqref{eq:3} is achievable due to the augmented dissipation facilitated by the non-self-adjoint operator $y\partial_x-\nu\Delta$, which represents the linear component of system \eqref{eq:3}. Operators of this nature have been extensively investigated by Hörmander (see \cite{Albritton.2022, Hörmander.1967}). In the case of the standard heat equation $\partial_t f-\nu\Delta f=0,$ its decay rate is  $t^{-1/2}$ ; however, for the drift diffusion equation $$\partial_t f+y\partial_x f-\nu\Delta f=0,$$ its decay rate is significantly faster at $t^{-1}$ compared with $t^{-1/2}.$ This intensified dissipation effect plays a crucial role in addressing the stability problem under consideration.
\end{remark}
\begin{remark}
	The transition threshold $\gamma\leq \frac{3}{4}$  may not be optimal, necessitating more precise estimations to enhance the outcomes.
	\end{remark}
	The utilization of Green's function in this paper is non-classical due to the presence of variable coefficients in the linear components of equation \eqref{eq:3}. Specifically, the variable coefficient $y$ plays a crucial role in enhancing dissipation, which distinguishes it from previous cases. Therefore, one must not only overcome challenges posed by these variable coefficients but also harness their advantages. The construction of Green's function for linearized equations with variable coefficients is highly intricate and often relies heavily on the specific structures of these coefficients. For references on constructing Green's functions for shock profiles and viscous rarefaction waves, please (see \cite{Liu.1997,Liu.2017,Liu.2010,Liu.2009}). 
	The precise formula provided in (see \cite{Marcus.1977}) is employed in this study to calculate the Green's function associated with the heat equation within a background Couette flow in two dimensions. Further details can be found in Section \ref{sec3}.
	
In Section 2, we present a review of several important lemmas that will be utilized in this paper. In Section 3, the Green's function for the linearization equation is provided and certain properties of this function are estimated. Finally, in Section 4, we establish nonlinear stability by utilizing the well-established properties of Green's function discussed in the previous Section.

Let us give some notations to end this section. Throughout this paper, we denote the space-time norm 
$$\|f\|_{L^p(t)L^q(\mathbb{R}^2)}=\left\|||f||_{L^q\left( \mathbb{R}^2\right)} \right\|_{L^p(\mathbb{R}^+)}.$$
 We always denote by $C$ a positive constant independent of $\nu$. We
also use $a \sim b$ for $C^{-1}b\leq a\leq Cb$
and $a \lesssim b$ for $a \leq Cb $ for some constants $ C > 0$
independent of $  \nu.$
	\section{Preliminaries}	
In this section, we give the following lemmas, which will be used frequently in later proofs. Detailed proof comes  from (see \cite{Bollob.1993,Tao.2006}).
We begin with the following abstract bootstrap argument will be used to establish the global stability of  solution.

\begin{lemma}\label{lem2.1}
	(see \cite{Tao.2006}). Let $I$ be a time interval, and for each $t \in I$ suppose we have two statements, a "hypothesis" $\mathbf{H}(t)$ and a "conclusion" $\mathbf{C}(t)$. Suppose we can verify the following four assertions:

(a) (Hypothesis implies conclusion) If $\mathbf{H}(t)$ is true for some time $t \in I$, then $\mathbf{C}(t)$ is also true for that time $t$.

(b) (Conclusion is stronger than hypothesis) If $\mathbf{C}(t)$ is true for some $t \in I$, then $\mathbf{H}\left(t^{\prime}\right)$ is true for all $t^{\prime} \in I$ in a neighbourhood of $t$.

(c) (Conclusion is closed) If $t_1, t_2, \ldots$ is a sequence of times in $ I$ which converges to another time $t \in I$. and $\mathbf{C}\left(t_n\right)$ is true for all $t_n$, then $\mathbf{C}(t)$ is true.

(d) (Base case) $\mathbf{H}(t)$ is true for at least one time $t \in I$.

Then $\mathbf{C}(t)$ is true for all $t \in I$.

\end{lemma}

	





	The subsequent lemma presents the Young's inequality for the kernel form.
\begin{lemma}\label{lem2.2}(Young's inequality )  Let $1 \leq p, q, r \leq \infty$ and $1+\frac{1}{r}=\frac{1}{q}+\frac{1}{p}.$ If the kernel $K(z,z')$ is a measurable function on $\mathbb{R}^d\times\mathbb{R}^d$ and satisfies
	$$
	\left\|K(\cdot,z')\right\|_{L^q}\leq A,\quad	\left\|K(z,\cdot)\right\|_{L^q}\leq B,
	$$
 we define  integral operator
	$$
	Tf(z)=\int_{\mathbb{R}^d}K(z,z')f(z')dz',
	$$
	then it holds that
	 $$
	 	\left\|	Tf(z)\right\|_{L^r}\leq C	\left\|	f\right\|_{L^p},
	 	$$
	 for	$f(z)\in L^{p}(\mathbb{R}^d),$	where $C=\max\left\lbrace A,B\right\rbrace$.
	Moreover, we have a much finer estimate:
	 $$
	\left\|	Tf(z)\right\|_{L^r}\leq A^{\frac{q}{r}}	B^{q-\frac{q}{p}}\left\|f\right\|_{L^p}.
	$$
	\end{lemma}
\section{Green's function }\label{sec3}
In this section, we will primarily equip ourselves with analytical tools for nonlinear stability in Section \ref{sec4}. These tools comprise fundamental estimates for Green's functions and lemmas for nonlinear stability. To fully utilize the precise formula of Green's function for 2-D as presented in \cite{Marcus.1977}, some modifications are required to our original equation. The construction of the Green's function for system \eqref{eq:3} involves time scaling $t\rightarrow\dfrac{t }{\nu},$ and rewrite the system as follows
\begin{equation}\label{eq:7}
	\left\{
	\begin{aligned}
		&\partial _t\omega+\frac{1}{\nu}y\partial_x\omega -\Delta \omega=-\frac{1}{\nu}\nabla\cdot (u \omega),\\
		&u=\nabla^{\perp}\Delta^{-1}\omega,\\
		&\omega(0,x,y)=\omega_{0}(x,y).\\
	\end{aligned}
	\right.
\end{equation}
The first step involves the consideration of the Green's function for the system \eqref{eq:7}
\begin{equation}\label{eq:8}
	\left\{
	\begin{aligned}
		&\partial _t\mathbb{G}+\frac{1}{\nu}y\partial_x\mathbb{G}-\Delta \mathbb{G}=0,\\
		&\mathbb{G}(x,y,0;y')=\delta(x,y-y').\\
	\end{aligned}
	\right.
\end{equation}
Here, we can write the $\mathbb{G}$ as folows
\begin{equation}\label{eq:9}
\begin{aligned}
\mathbb{G}\left(x, y, t; y'\right)&=\frac{1}{4 \pi t}\left(1+\frac{1}{12} \nu^{-2} t^2\right)^{-1 / 2}\cdot\\ 
&\exp \left(\frac{-\left(x-\frac{ t}{2\nu}\left(y+y'\right)\right)^2-\left(1+\frac{1}{12} \nu^{-2} t^2\right)\left(y-y'\right)^2}{4 t\left(1+\frac{1}{12}  \nu^{-2} t^2\right)}\right),
\end{aligned}
\end{equation}
where we use the precise formula given in (see \cite{Marcus.1977}) of Green’s function for the heat equation in the background of a
2-D Couette flow. A more streamlined version of Green's function is presented in \cite{Deng.2023}.
 Compared with the classical heat kernel in 2D $$\mathbb{H}(x, y, t) \equiv \frac{1}{4 \pi t} \exp \left(\frac{-x^2-y^2}{4 t}\right),$$ there is an extra factor $\left(1+\frac{1}{12}\nu^{-2} t^2\right)^{-1 / 2}$ in $\mathbb{G}\left(x, y, t ; y'\right)$ given by \eqref{eq:9}. This factor is an evidence of dissipation enhancement in $L^p (p>1)$ norm and also suggests that in the linear level for long time the solution should become sufficiently small since
$$
\left(1+\frac{1}{12} \nu^{-2} t^2\right)^{-1 / 2} \ll 1, \quad \text { for } t\gg\nu .
$$
In the nonlinear regime, the time evolution of the solution becomes more intricate due to the presence of nonlinear effects. It requires a longer duration for the solution to sufficiently diminish. An intriguing aspect in analyzing nonlinear  singularity at $t=s$  in time coincides with an additional decaying factor$\left(1+\frac{1}{12} \nu^{-2} t^2\right)^{-1 / 2}$, resulting in non-integrable singularities appearing in the nonlinear terms. We address this issue by making another crucial observation that $t=O\left(\nu^{1 / 2}\right)$ represents a balance point between these two effects: the singularity in time and the extra decaying factor. Accordingly, we describe the solution within different temporal regions based on this balancing point (see Section \ref{sec4}). The utilization of Green's function method employed in this study is proved to be a powerful tool for providing a clear depiction of space-time structures within our solution. The detailed information obtained here enables us to capture enhanced dissipation effects across all frequencies. Unfortunately, Green's functions for dissipative equations involving other shear flows remain unclear at present. However, we hope that our developed methodology can shedlight on such cases as well, leaving them open for future investigation.

 
 The following is an estimation of the Green's function, which plays a pivotal role in our proof.

\begin{lemma}\label{lem3.1}
The following estimates for the Green's function are available for $p\geq 1$
\begin{equation}
	\begin{aligned}
&	\|\mathbb{G}(x-\cdot,y,t;\cdot)\|_{L^p}\leq Ct^{-(1-\frac{1}{p})}\left(1+\frac{1}{12} \nu^{-2} t^2\right)^{-\frac{1}{2}(1-\frac{1}{p})},\\
&\|\mathbb{G}(\cdot-x',\cdot,t;y')\|_{L^p}\leq Ct^{-(1-\frac{1}{p})}\left(1+\frac{1}{12} \nu^{-2} t^2\right)^{-\frac{1}{2}(1-\frac{1}{p})}.
	\end{aligned}
\end {equation}
\end{lemma}
\begin {proof}
The expression of the Green's function in \eqref{eq:9} indicates that 
\[
	\begin{aligned}
&	\|\mathbb{G}(x-\cdot,y,t;\cdot)\|_{L^p}^p\\
&=\left( \frac{1}{4 \pi t}\right) ^p\left(1+\frac{1}{12} \nu^{-2} t^2\right)^{-p / 2} \cdot\\
& \int_{\mathbb{R}^2}\exp \left(\frac{-p\left(x-x'-\frac{ t}{2\nu}\left(y+y'\right)\right)^2}{4 t\left(1+\frac{1}{12}  \nu^{-2} t^2\right)}\right) \exp \left( \dfrac{-p(y-y')^2}{4t}\right) dx'dy'\\
&=\left( \frac{1}{4 \pi t}\right) ^p\left(1+\frac{1}{12} \nu^{-2} t^2\right)^{-p / 2} \cdot\\
& \int_{\mathbb{R}} \exp \left( \dfrac{-p(y-y')^2}{4t}\right)dy'\int_{\mathbb{R}}\exp \left(\frac{-p\left(x-x'-\frac{ t}{2\nu}\left(y+y'\right)\right)^2}{4 t\left(1+\frac{1}{12}  \nu^{-2} t^2\right)}\right) dx'\\
&\leq C t^{-(p-1)}\left(1+\frac{1}{12} \nu^{-2} t^2\right)^{-p / 2+\frac{1}{2}},
	\end{aligned}
\]
therefore, we have obtained
$$\|\mathbb{G}(x-\cdot,y,t;\cdot)\|_{L^p}\leq Ct^{-(1-\frac{1}{p})}\left(1+\frac{1}{12} \nu^{-2} t^2\right)^{-\frac{1}{2}(1-\frac{1}{p})}.$$
The procedure is identical to that of
$$\|\mathbb{G}(\cdot-x',\cdot,t;y')\|_{L^p}\leq Ct^{-(1-\frac{1}{p})}\left(1+\frac{1}{12} \nu^{-2} t^2\right)^{-\frac{1}{2}(1-\frac{1}{p})}.$$
This completes the proof of the Lemma \ref{lem3.1}.
\end{proof}
The subsequent proof will also employ an estimates of the derivative of Green's function. The estimation of the first derivative of Green's function suffices in this paper due to the requirement for evidence. 
The Green's function, derived from the expression \eqref{eq:9}, can be regarded as a kernel function. However, when it operates on a function, it no longer assumes the simple convolution form and lacks symmetry in relation to the variables $(x,y)$ and $(x',y').$ The Green function exhibits distinct characteristics between the variables $(x,x')$ and $(y,y').$  The estimation of the derivative is required for both $(x,x') $ and $(y,y'). $ The initial step involves estimating the derivative with respect to $(x,x').$
	\begin{lemma}\label{lem3.2}
			The following derivative estimate for the Green's function is at our disposal for $p\geq 1$
			\begin{equation}\label{eq:11}
				\begin{aligned}
					&	\|\partial_{x,x'}\mathbb{G}(x-\cdot,y,t;\cdot)\|_{L^p}\leq Ct^{-(1-\frac{1}{p})-\frac{1}{2}}\left(1+\frac{1}{12} \nu^{-2} t^2\right)^{-\frac{1}{2}(1-\frac{1}{p})-\frac{1}{2}},\\
		&	\|\partial_{x,x'}\mathbb{G}(\cdot-x',\cdot,t;y')\|_{L^p}\leq Ct^{-(1-\frac{1}{p})-\frac{1}{2}}\left(1+\frac{1}{12} \nu^{-2} t^2\right)^{-\frac{1}{2}(1-\frac{1}{p})-\frac{1}{2}}.
				\end{aligned}
			\end{equation}
			\end{lemma}
\begin {proof}
The expression of the Green's function in \eqref{eq:9} indicates that
$$
\begin{aligned}
	&	\|\partial_{x}\mathbb{G}(x-\cdot,y,t;\cdot)\|_{L^p}^p\\
	&=\left( \frac{1}{4 \pi t}\right) ^p\left(1+\frac{1}{12} \nu^{-2} t^2\right)^{-p / 2} \cdot\\
	& \int_{\mathbb{R}^2}\exp \left(\frac{-p\left(x-x'-\frac{ t}{2\nu}\left(y+y'\right)\right)^2}{4 t\left(1+\frac{1}{12}  \nu^{-2} t^2\right)}\right) \exp \left( \dfrac{-p(y-y')^2}{4t}\right) \left( \frac{x-x'-\frac{ t}{2\nu}\left(y+y'\right)}{2 t\left(1+\frac{1}{12}  \nu^{-2} t^2\right)}\right)^p dx'dy'\\
	&=\left( \frac{1}{4 \pi t}\right) ^{p+p/2}\left(1+\frac{1}{12} \nu^{-2} t^2\right)^{-p }\cdot \\
	& \int_{\mathbb{R}} \exp \left( \dfrac{-p(y-y')^2}{4t}\right)dy'\int_{\mathbb{R}}\exp \left(\frac{-p\left(x-x'-\frac{ t}{2\nu}\left(y+y'\right)\right)^2}{4 t\left(1+\frac{1}{12}  \nu^{-2} t^2\right)}\right)\left(  \frac{x-x'-\frac{ t}{2\nu}\left(y+y'\right)}{\sqrt{2 t\left(1+\frac{1}{12}  \nu^{-2} t^2\right)}}\right)^pdx'\\
	&\leq C \left( \frac{1}{4 \pi t}\right) ^{p+p/2-1}\left(1+\frac{1}{12} \nu^{-2} t^2\right)^{-p  +1/2},
\end{aligned}
$$
so, we obtain 
	$$\|\partial_{x}\mathbb{G}(x-\cdot,y,t;\cdot)\|_{L^p}\leq Ct^{-(1-\frac{1}{p})-\frac{1}{2}}\left(1+\frac{1}{12} \nu^{-2} t^2\right)^{-\frac{1}{2}(1-\frac{1}{p})-\frac{1}{2}}.$$
Similarly, we obtain 
	$$\|\partial_{x'}\mathbb{G}(x-\cdot,y,t;\cdot)\|_{L^p}\leq Ct^{-(1-\frac{1}{p})-\frac{1}{2}}\left(1+\frac{1}{12} \nu^{-2} t^2\right)^{-\frac{1}{2}(1-\frac{1}{p})-\frac{1}{2}}.$$
	This completes the proof of the Lemma \ref{lem3.2}.
	\end{proof}
		Next, the objective is to estimate the derivative for Green's function with respect to the $(y, y').$
\begin{lemma}\label{lem3.3}
		The following derivative estimate for the Green's function is at our disposal for $p\geq 1$
		\begin{equation}\label{eq:12}
			\begin{aligned}
				&	\|\partial_{y,y'}\mathbb{G}(x-\cdot,y,t;\cdot)\|_{L^p}\leq Ct^{-(1-\frac{1}{p})-\frac{1}{2}}\left(1+\frac{1}{12} \nu^{-2} t^2\right)^{-\frac{1}{2}(1-\frac{1}{p})},\\
				&	\|\partial_{y,y'}\mathbb{G}(\cdot-x',\cdot,t;y')\|_{L^p}\leq Ct^{-(1-\frac{1}{p})-\frac{1}{2}}\left(1+\frac{1}{12} \nu^{-2} t^2\right)^{-\frac{1}{2}(1-\frac{1}{p})}.
			\end{aligned}
		\end{equation}
		\end{lemma}
\begin{proof}
The $\partial_{y}\mathbb{G}(x-x',y,t;y')$ can be derived through direct computation,  
\[
\begin{aligned}
&\partial_{y}\mathbb{G}(x-x',y,t;y')=\\
&\frac{1}{4 \pi t}\left(1+\frac{1}{12} \nu^{-2} t^2\right)^{-1 / 2}\cdot\\
&\exp \left(\frac{-\left(x-x'-\frac{ t}{2\nu}\left(y+y'\right)\right)^2}{4 t\left(1+\frac{1}{12}  \nu^{-2} t^2\right)}-\dfrac{(y-y')^2}{4t}\right) \left(\frac{t}{\nu} \frac{\left( x-x'-\frac{ t}{2\nu}\left(y+y'\right)\right) }{4t\left(1+\frac{1}{12}  \nu^{-2} t^2\right)}\right)\\
&+\frac{1}{4 \pi t}\left(1+\frac{1}{12} \nu^{-2} t^2\right)^{-1 / 2}\cdot\\
&\exp \left(\frac{-\left(x-x'-\frac{ t}{2\nu}\left(y+y'\right)\right)^2}{4 t\left(1+\frac{1}{12}  \nu^{-2} t^2\right)}\right) \exp \left( \dfrac{-(y-y')^2}{4t}\right)\left( \frac{-2\left( y-y'\right) }{4t}\right) \\
&=M_1+M_2,
\end{aligned}
\]
where we divide it into two parts.
The first step involves estimating  $M_1$:
\begin{equation}\label{eq:13}
\begin{aligned}
	&	\|M_1(x-\cdot,y,t;\cdot)\|_{L^p}^p\\
	&\leq C\left( \frac{1}{4 \pi t}\right) ^{p+\frac{p}{2}}\left(1+\frac{1}{12} \nu^{-2} t^2\right)^{-p / 2} \cdot\\
	& \int_{\mathbb{R}^2}\exp \left(\frac{-p\left(x-x'-\frac{ t}{2\nu}\left(y+y'\right)\right)^2}{4 t\left(1+\frac{1}{12}  \nu^{-2} t^2\right)}\right) \exp \left( \dfrac{-p(y-y')^2}{4t}\right) \left(  \frac{x-x'-\frac{ t}{2\nu}\left(y+y'\right)}{\sqrt{2 t\left(1+\frac{1}{12}  \nu^{-2} t^2\right)}}\right)^pdx'dy'\\
	&\leq C \left( \frac{1}{4 \pi t}\right) ^{p+p/2-1}\left(1+\frac{1}{12} \nu^{-2} t^2\right)^{-p/2 +1/2}.
\end{aligned}
\end{equation}
We proceed to estimate the second dynamic form factor $M_2$:
\begin{equation}\label{eq:14}
\begin{aligned}
	&	\|M_2(x-\cdot,y,t;\cdot)\|_{L^p}^p\\
	&=\left( \frac{1}{4 \pi t}\right) ^p\left(1+\frac{1}{12} \nu^{-2} t^2\right)^{-p / 2} \cdot\\
	& \int_{\mathbb{R}^2}\exp \left(\frac{-p\left(x-x'-\frac{ t}{2\nu}\left(y+y'\right)\right)^2}{4 t\left(1+\frac{1}{12}  \nu^{-2} t^2\right)}\right) \exp \left( \dfrac{-p(y-y')^2}{4t}\right) \left( \frac{\left(y-y'\right)}{ \sqrt{t}}\right) ^p\left( \frac{1}{2\sqrt{t}}\right) ^pdx'dy'\\
	&\leq C \left( \frac{1}{4 \pi t}\right) ^{p+p/2-1}\left(1+\frac{1}{12} \nu^{-2} t^2\right)^{-p/2 +1/2}.
	\end{aligned}
\end{equation}
By combining equations \eqref{eq:13} and \eqref{eq:14}, we can derive the desired outcome.

Similarly, we obtain
$$\|\partial_{y'}\mathbb{G}(x-\cdot,y,t;\cdot)\|_{L^p}\leq Ct^{-(1-\frac{1}{p})-\frac{1}{2}}\left(1+\frac{1}{12} \nu^{-2} t^2\right)^{-\frac{1}{2}(1-\frac{1}{p})}.$$
This completes the proof of the Lemma \ref{lem3.3}.
\end{proof}
\section{Nonlinear stability}\label{sec4}
In this section, we establish the proof of  Theorem \ref{thm1}. The global existence of smooth solution is well-known  for the data $$\omega_0\in H^1(\mathbb{R}^2)\cap L^{1}(\mathbb{R}^2).$$  The primary focus of Theorem \ref{thm1} lies in the estimation of stability
\begin{equation}\label{eq:15}
	||\omega(t)||_{L^2}\leq Cc_0(1+t)^{-1}\nu^{\frac{3}{4}}.
\end{equation}
 We aim to  use the bootstrap argument to establish the estimates \eqref{eq:15} if the initial data satisfies 
 \begin{equation}\label{eq:16}
  \|\omega_{0}\|_{H^{1}\cap L^1}=\epsilon\leq c_0\nu^{\frac{3}{4}},
\end{equation}
with sufficiently small constant $c_0.$

The bootstrap argument is applied by considering a fixed number, denoted as $\delta>0$. For any $t \in[0, T]$, we begin with the hypothesis that 
\begin{equation}\label{eq:17}
\mathbf{H}(t):	||\omega(t)||_{L^2}\leq \delta(1+t)^{-1}\epsilon,
\end{equation}
and the conclusion, denoted as
\begin{equation}\label{eq:18}
\mathbf{C}(t):	||\omega(t)||_{L^2}\leq\frac{\delta}{2} (1+t)^{-1}\epsilon.
\end{equation}
The conditions $b-d$ stated in Lemma \ref{lem2.1} are satisfied, and we only need to verify condition $a$ under the assumption of \eqref{eq:16}.
\subsection{Stability estimate in short time $t\leq2 \nu^{\frac{1}{2}}$ }\
\begin{proposition}\label{pro4.1}
	Assuming that $t\leq2 \nu^{\frac{1}{2}}$ and the existence of a sufficiently small constant $c_0$, which is independent of $\nu$, such that $$\|\omega_{0}\|_{H^{1}\cap L^1}=\epsilon\leq c_0\nu^{3/4},$$then based on the bootstrap hypothesis \eqref{eq:17}, we can infer$$\|\omega(t)\|_{L^2}\leq \frac{\delta}{2}(1+t)^{-1}\epsilon.$$
\end{proposition}
\begin{proof}
By Duhamel principle and  intergration by parts, the solution of \eqref{eq:3} can be expressed as
\begin{equation}\label{eq:19}
	\begin{aligned}
		&\omega(x,y,t)=\int_{\mathbb{R}^2}\mathbb{G}(x-x',y,t;y')\omega_{0}(x',y')dx'dy'\\
		&+\nu^{-1}\int_{0}^{t}\int_{\mathbb{R}^2}-\mathbb{G}(x-x',y,t-s;y')\nabla\cdot\left( u\omega\right) (x',y',s)dx'dy'ds\\
		&=\int_{\mathbb{R}^2}\mathbb{G}(x-x',y,t;y')\omega_{0}(x',y')dx'dy'\\
		&+\nu^{-1}\int_{0}^{t}\int_{\mathbb{R}^2}\nabla_{x',y'}\mathbb{G}(x-x',y,t-s;y')\left( u\omega\right) (x',y',s)dx'dy'ds\\
		&=J_1+J_2.\\
	\end{aligned}
\end{equation}
The term $J_1$ is initially estimated using Lemma \ref{lem2.2} and Lemma \ref{lem3.1}, thereby establishing the following result
\begin{equation}\label{eq:20}
	\begin{aligned}
		\| J_1\|_{L^2}&=\left\|\int_{\mathbb{R}^2}\mathbb{G}(x-x',y,t;y')\omega_{0}(x',y')dx'dy'\right\|_{L^2}\leq ||\mathbb{G}||_{L^2}\|\omega_{0}\|_{L^1}\\
		& \leq C_1\nu^{\frac{1}{2}} t^{-1}\|\omega_{0}\|_{L^1},
	\end{aligned}
\end{equation}
or 
\begin{equation}\label{eq:21}
	\begin{aligned}
		\| J_1\|_{L^2}&=\left\|\int_{\mathbb{R}^2}\mathbb{G}(x-x',y,t;y')\omega_{0}(x',y')dx'dy'\right\|_{L^2}\leq \|\mathbb{G}\|_{L^1}\|\omega_{0}\|_{L^2}\\
		& \leq C_2\|\omega_{0}\|_{L^2},
	\end{aligned}
\end{equation}
and the combination of equations \eqref{eq:20} and \eqref{eq:21} yields the following result
\begin{equation}\label{eq:22}
	\| J_1\|_{L^2}\leq C_3(1+t)^{-1}\epsilon.
\end{equation}
The next step is to divide $J_2 $ into two separate components
\begin{equation}\label{eq:23}
	\begin{aligned}
		J_2&=\nu^{-1}\int_{0}^{t}\int_{\mathbb{R}^2}\nabla_{x',y'}\mathbb{G}(x-x',y,t-s;y')\left( u\omega\right) (x',y',s)dx'dy'ds\\
		&=\nu^{-1}\int_{t/2}^{t}\int_{\mathbb{R}^2}\nabla_{x',y'}\mathbb{G}(x-x',y,t-s;y')\left( u\omega\right) (x',y',s)dx'dy'ds\\
		&+\nu^{-1}\int_{0}^{t/2}\int_{\mathbb{R}^2}\nabla_{x',y'}\mathbb{G}(x-x',y,t-s;y')\left( u\omega\right) (x',y',s)dx'dy'ds\\
		&=J_{21}+J_{22}.
	\end{aligned}
\end{equation}
The term $J_{21}$ is supported by Lemma \ref{lem2.2}, Lemma \ref{lem3.2},  Lemma \ref{lem3.3} and the bootstrap hypothesis \eqref{eq:17}, then we have
\begin{equation}\label{eq:24}
	\begin{aligned}
		\|J_{21}\|_{L^2}&=\nu^{-1}\left\|\int_{t/2}^{t}\int_{\mathbb{R}^2}\nabla_{x',y'}\mathbb{G}(x-x',y,t-s;y')\left( u\omega\right) (x',y',s)dx'dy'ds\right\|_{L^2}\\
	&	\leq \nu^{-1}\int_{t/2}^{t}\left\|\int_{\mathbb{R}^2}\nabla_{x',y'}\mathbb{G}(x-x',y,t-s;y')\left( u\omega\right) (x',y',s)dx'dy'\right\|_{L^2}ds\\
		&	\leq \nu^{-1}\int_{t/2}^{t}\|\nabla_{x',y'}\mathbb{G}(x-\cdot,y,t-s;\cdot)\|_{L^{10/9}}\|\left( u\omega\right) (s)\|_{L^{\frac{5}{3}}}ds\\
		&	\leq C\nu^{-9/10}\int_{t/2}^{t}(t-s)^{-\frac{5}{2}+\frac{9}{5}}\| \omega(s)\|_{L^2}\|u(s)\|_{L^{10}}ds.\\
		&	\leq C\delta(1+t)^{-1}\epsilon^2\nu^{-3/4},
	\end{aligned}
\end{equation}
where the condition $t\leq2 \nu^{\frac{1}{2}}$ was employed in this case. 

In relation to term $J_{22},$ we consulted Lemma \ref{lem2.2}, Lemma \ref{lem3.2} and Lemma \ref{lem3.3} for guidance, while incorporating the bootstrap hypothesis \eqref{eq:17} through application of the Gagliardo-Nirenberg inequality, resulting in
\begin{equation}\label{eq:25}
	\begin{aligned}
		\|J_{22}\|_{L^2}
		&	\leq \nu^{-1}\int_{0}^{t/2}\left\|\int_{\mathbb{R}^2}\nabla_{x',y'}\mathbb{G}(x-x',y,t-s;y')\left( u\omega\right) (x',y',s)dx'dy'\right\|_{L^2}ds\\
		&	\leq C\nu^{-1}\int_{0}^{t/2}(t-s)^{-\frac{2}{5}-\frac{1}{2}}\left(1+\frac{1}{12} \nu^{-2} (t-s)^2\right)^{-\frac{1}{5}}\|\left( u\omega\right) (s)\|_{L^{\frac{10}{9}}}ds\\
		&	\leq C\nu^{-3/5}\int_{0}^{t/2}(t-s)^{-\frac{4}{5}-\frac{1}{2}}\| \omega(s)\|_{L^2}\|u(s)\|_{L^{\frac{5}{2}}}ds\\
		&	\leq C\nu^{-3/5}t^{-\frac{13}{10}}\int_{0}^{t/2}\| \omega(s)\|_{L^2}\|u(s)\|_{L^{2}}^{\frac{9}{10}}\|\nabla^2u(s)\|_{L^{2}}^{\frac{1}{10}}ds\\
		&	\leq C\delta\epsilon^2\nu^{-3/5}t^{-\frac{13}{10}}\int_{0}^{t/2}(1+s)^{-1}ds\\
		&\leq C\delta\epsilon^2\nu^{-1/10}t^{-\frac{13}{10}},\\
	\end{aligned}
\end{equation}
where the condition $t\leq2 \nu^{\frac{1}{2}}$ was utilized. 

The bounds of $J_{22}$ can be derived using the same methodology
 \begin{equation}\label{eq:26}
	\begin{aligned}
		\|J_{22}\|_{L^2}
		&	\leq C\nu^{-1}\int_{0}^{t/2}\|\nabla_{x',y'}\mathbb{G}(x-\cdot,y,t-s;\cdot)\|_{L^{\frac{4}{3}}}\|\left( u\omega\right) (s)\|_{L^{\frac{4}{3}}}ds\\
		&	\leq C\nu^{-3/4}t^{-1}\int_{0}^{t/2}\| \omega(s)\|_{L^2}\|u(s)\|_{L^{2}}^{\frac{3}{4}}\|\nabla^2u(s)\|_{L^{2}}^{\frac{1}{4}}ds\\
		&	\leq C\delta\epsilon^2\nu^{-3/4}t^{-1}\int_{0}^{t/2}(1+s)^{-1}ds\\
		&	\leq C\delta\epsilon^2\nu^{-3/4}.\\
	\end{aligned}
\end{equation}
The combination of \eqref{eq:25} and \eqref{eq:26}, in conjunction with, allows us to derive

\begin{equation}\label{eq:27}
	\|J_{22}\|_{L^2}\leq C\delta\epsilon^2\nu^{-3/4}(1+t)^{-1}.
\end{equation}
If $\delta$  is suitably selected and $c_0$ is sufficiently small, meaning
 $$C_3\leq\frac{\delta}{2}, \quad \epsilon\leq c_0\nu^{3/4}.$$
 By combining equations \eqref{eq:19}, \eqref{eq:22}, \eqref{eq:24} and \eqref{eq:27}, as referenced, then we obtain the conclusion that
 \begin{equation}\label{4.13}
 	\begin{aligned}
  \|\omega(t)\|_{L^2}&\leq \|J_1\|_{L^2}+\|J_2||_{L^2}\leq\|J_1\|_{L^2}+\|J_{21}\|_{L^2}+\|J_{22}\|_{L^2}\\
  &\leq C_3(1+t)^{-1}\epsilon+C\delta\epsilon^2\nu^{-3/4}(1+t)^{-1}\\
  &\leq \frac{\delta}{2}(1+t)^{-1}\epsilon.
  	\end{aligned}
 \end{equation}
This completes the proof of the Proposition \ref{pro4.1}.
 \end{proof}
\subsection{Stability estimate in  time $t\geq2 \nu^{\frac{1}{2}}$ }\
\begin{proposition}\label{pro4.2}
	Assuming that $t\geq2 \nu^{\frac{1}{2}}$ and the existence of a sufficiently small constant $c_0$, which is independent of $\nu$, such that $$\|\omega_{0}\|_{H^{1}\cap L^1}=\epsilon\leq c_0\nu^{3/4},$$then based on the bootstrap hypothesis \eqref{eq:17}, we can infer$$\|\omega(t)\|_{L^2}\leq \frac{\delta}{2}(1+t)^{-1}\epsilon.$$
\end{proposition}
\begin{proof}
The solution of equation \eqref{eq:3} can be expressed using the Duhamel principle and integration by parts
\begin{equation}\label{eq:29}
	\begin{aligned}
		&\omega(x,y,t)=\int_{\mathbb{R}^2}\mathbb{G}(x-x',y,t;y')\omega_{0}(x',y')dx'dy'\\
		&+\nu^{-1}\int_{0}^{t}\int_{\mathbb{R}^2}-\mathbb{G}(x-x',y,t-s;y')\nabla\cdot\left( u\omega\right) (x',y',s)dx'dy'ds\\
		&=\int_{\mathbb{R}^2}\mathbb{G}(x-x',y,t;y')\omega_{0}(x',y')dx'dy'\\
		&+\nu^{-1}\int_{0}^{t}\int_{\mathbb{R}^2}\nabla_{x',y'}\mathbb{G}(x-x',y,t-s;y')\left( u\omega\right) (x',y',s)dx'dy'ds\\
		&=I_1+I_2.\\
	\end{aligned}
\end{equation}
The same procedure as $ J_{1}$ can be applied to the term $I_1$, then we obtain
 \begin{equation}\label{eq:30}
 \| I_1\|_{L^2}\leq C_3(1+t)^{-1}\epsilon.
\end{equation}
The term $ I_2$, which is distinct from $J_2$, can be divided into three segments by a constant $\nu^{\theta}$ that is determined by the subsequent proof
 \begin{equation}\label{eq:31}
 		\begin{aligned}
 	I_2&=\nu^{-1}\int_{0}^{t}\int_{\mathbb{R}^2}\nabla_{x',y'}\mathbb{G}(x-x',y,t-s;y')\left( u\omega\right) (x',y',s)dx'dy'ds\\
 	&=\nu^{-1}\int_{t-\nu^\theta}^{t}\int_{\mathbb{R}^2}\nabla_{x',y'}\mathbb{G}(x-x',y,t-s;y')\left( u\omega\right) (x',y',s)dx'dy'ds\\
 	&+\nu^{-1}\int_{t/2}^{t-\nu^{\theta}}\int_{\mathbb{R}^2}\nabla_{x',y'}\mathbb{G}(x-x',y,t-s;y')\left( u\omega\right) (x',y',s)dx'dy'ds\\
 	&+\nu^{-1}\int_{0}^{\frac{t}{2}}\int_{\mathbb{R}^2}\nabla_{x',y'}\mathbb{G}(x-x',y,t-s;y')\left( u\omega\right) (x',y',s)dx'dy'ds\\
 	&=I_{21}+I_{22}+I_{23}.\\
 		\end{aligned}
\end{equation}
The term $ I_{21}$ can be deduced from Lemma \ref{lem2.2}, Lemma \ref{lem3.2} and Lemma \ref{lem3.3} that
 \begin{equation}\label{eq:32}
 		\begin{aligned}
 		 \|I_{21}\|_{L^2}&=\nu^{-1}\left\|\int_{t-\nu^\theta}^{t}\int_{\mathbb{R}^2}\nabla_{x',y'}\mathbb{G}(x-x',y,t-s;y')\left( u\omega\right) (x',y',s)dx'dy'ds\right\|_{L^2}\\
 		  &	\leq \nu^{-1}\int_{t-\nu^\theta}^{t}\left\|\int_{\mathbb{R}^2}\nabla_{x',y'}\mathbb{G}(x-x',y,t-s;y')\left( u\omega\right) (x',y',s)dx'dy'\right\|_{L^2}ds\\
 		    &	\leq \nu^{-1}\int_{t-\nu^\theta}^{t}\|\nabla_{x',y'}\mathbb{G}(x-\cdot,y,t-s;\cdot)\|_{L^{p_1}}\|\left( u\omega\right) (s)\|_{L^{\frac{2p_1}{3p_1-2}}}ds\\
 		    &	\leq C\nu^{-1}\int_{t-\nu^\theta}^{t}(t-s)^{-(1-\frac{1}{p_1})-\frac{1}{2}}\left(1+\frac{1}{12} \nu^{-2} (t-s)^2\right)^{-\frac{1}{2}(1-\frac{1}{p_1})}\|\left( u\omega\right) (s)\|_{L^{\frac{2p_1}{3p_1-2}}}ds\\
 		      &	\leq C\nu^{-1}\int_{t-\nu^\theta}^{t}(t-s)^{-(1-\frac{1}{p_1})-\frac{1}{2}}\left(1+\frac{1}{12} \nu^{-2} (t-s)^2\right)^{-\frac{1}{2}(1-\frac{1}{p_1})}\|\left( u\omega\right) (s)\|_{L^{\frac{2p_1}{3p_1-2}}}ds\\
 		      &	\leq C\nu^{-1/p_1}\int_{t-\nu^\theta}^{t}(t-s)^{-2(1-\frac{1}{p_1})-\frac{1}{2}}\|\left( u\omega\right) (s)\|_{L^{\frac{2p_1}{3p_1-2}}}ds\\
 		       &	\leq C\nu^{-1/p_1}\int_{t-\nu^\theta}^{t}(t-s)^{-2(1-\frac{1}{p_1})-\frac{1}{2}}\| \omega(s)\|_{L^2}\|u(s)\|_{L^{\frac{p_1}{p_1-1}}}ds.
 		 	\end{aligned}
 		\end{equation}
 		Based on the bootstrap hypothesis \eqref{eq:17}, by ensuring the integrability of time at $s=t,$  and considering the requirement $1\leq p_1<4/3$, we can deduce that
 \begin{equation}\label{eq:33}
	\begin{aligned}
		\|I_{21}\|_{L^2}
 &	\leq C\delta(1+t)^{-1}\epsilon\nu^{-1/p_1}\int_{t-\nu^\theta}^{t}(t-s)^{-2(1-\frac{1}{p_1})-\frac{1}{2}}\|u(s)\|_{L^{\frac{p_1}{p_1-1}}}ds\\
  &	\leq C\delta(1+t)^{-1}\epsilon\nu^{-1/p_1}\int_{t-\nu^\theta}^{t}(t-s)^{-2(1-\frac{1}{p_1})-\frac{1}{2}}\|u(s)\|_{L^2}^{\frac{3}{2}-\frac{1}{p_1}}\|\nabla^2u(s)\|_{L^2}^{\frac{1}{p_1}-\frac{1}{2}}ds\\
  &	\leq C\delta(1+t)^{-1}\epsilon^2\nu^{-1/p_1+2\theta/p_1-3\theta/2}.
\end{aligned}
\end{equation}
The next step involves estimating the term $I_{22}$ as follows
\begin{equation}\label{eq:34}
	\begin{aligned}
		\|I_{22}\|_{L^2}&=\nu^{-1}\left\|\int_{t/2}^{t-\nu^\theta}\int_{\mathbb{R}^2}\nabla_{x',y'}\mathbb{G}(x-x',y,t-s;y')\left( u\omega\right) (x',y',s)dx'dy'ds\right\|_{L^2}\\
		&	\leq \nu^{-1}\int_{t/2}^{t-\nu^\theta}\left\|\int_{\mathbb{R}^2}\nabla_{x',y'}\mathbb{G}(x-x',y,t-s;y')\left( u\omega\right) (x',y',s)dx'dy'\right\|_{L^2}ds\\
		&	\leq \nu^{-1}\int_{t/2}^{t-\nu^\theta}\|\nabla_{x',y'}\mathbb{G}(x-\cdot,y,t-s;\cdot)\|_{L^{p_2}}\|\left( u\omega\right) (s)\|_{L^{\frac{2p_2}{3p_2-2}}}ds\\
		&	\leq C\nu^{-1/p_2}\int_{t/2}^{t-\nu^\theta}(t-s)^{-2(1-\frac{1}{p_2})-\frac{1}{2}}\|\left( u\omega\right) (s)\|_{L^{\frac{2p_2}{3p_2-2}}}ds\\
		&	\leq C\nu^{-1/p_2}\int_{t/2}^{t-\nu^\theta}(t-s)^{-2(1-\frac{1}{p_2})-\frac{1}{2}}\| \omega(s)\|_{L^2}\|u(s)\|_{L^{\frac{p_2}{p_2-1}}}ds.\\
	\end{aligned}
\end{equation}
Based on the bootstrap hypothesis \eqref{eq:17} and the requirement that $4/3< p_2\leq2$, it can be inferred that 
\begin{equation}\label{eq:35}
	\begin{aligned}
		\|I_{22}\|_{L^2}
		&	\leq C\delta(1+t)^{-1}\epsilon\nu^{-1/p_2}\int_{t/2}^{t-\nu^\theta}(t-s)^{-2(1-\frac{1}{p_2})-\frac{1}{2}}\|u(s)\|_{L^{\frac{p_2}{p_2-1}}}ds\\
		&	\leq C\delta(1+t)^{-1}\epsilon\nu^{-1/p_2}\int_{t-1}^{t-\nu^\theta}(t-s)^{-2(1-\frac{1}{p_2})-\frac{1}{2}}\|u(s)\|_{L^2}^{\frac{3}{2}-\frac{1}{p_2}}\|\nabla^2u(s)\|_{L^2}^{\frac{1}{p_2}-\frac{1}{2}}ds\\
		&	\leq C\delta(1+t)^{-1}\epsilon^2\nu^{-1/p_2+2\theta/p_2-3\theta/2}.\\
	\end{aligned}
\end{equation}
The final step involves estimating the item $I_{23},$
 \begin{equation}\label{eq:36}
	\begin{aligned}
		\|I_{23}\|_{L^2}&=\nu^{-1}\left\|\int_{0}^{t/2}\int_{\mathbb{R}^2}\nabla_{x',y'}\mathbb{G}(x-x',y,t-s;y')\left( u\omega\right) (x',y',s)dx'dy'ds\right\|_{L^2}\\
		&\leq \nu^{-1}\int_{0}^{t/2}\left\|\int_{\mathbb{R}^2}\nabla_{x',y'}\mathbb{G}(x-x',y,t-s;y')\left( u\omega\right) (x',y',s)dx'dy'\right\|_{L^2}ds\\
		&	\leq C\nu^{-1}\int_{0}^{t/2}\|\nabla_{x',y'}\mathbb{G}(x-\cdot,y,t-s;\cdot)\|_{L^{5/3}}\|\left( u\omega\right) (s)\|_{L^{\frac{10}{9}}}ds\\
		&	\leq C\nu^{-3/5}\int_{0}^{t/2}(t-s)^{-\frac{4}{5}-\frac{1}{2}}\|\left( u\omega\right) (s)\|_{L^{\frac{10}{9}}}ds\\
		&	\leq C\nu^{-3/5}t^{-\frac{13}{10}}\int_{0}^{t/2}\| \omega(s)\|_{L^2}\|u(s)\|_{L^{2}}^{\frac{9}{10}}\|\nabla^2u(s)\|_{L^{2}}^{\frac{1}{10}}ds\\
		&	\leq C\delta\epsilon^2\nu^{-3/5}t^{-\frac{13}{10}}\int_{0}^{t/2}(1+s)^{-1}ds\\
			&	\leq C\delta\epsilon^2\nu^{-3/4}t^{-1},\\
	\end{aligned}
\end{equation}
where we have used the $t\geq 2\nu^{\frac{1}{2}}$ and the value $\nu$ sufficiently small.
We proceed to establish the boundedness of the element $ I_{23}$
 \begin{equation}\label{eq:37}
	\begin{aligned}
		\|I_{23}\|_{L^2}
		&	\leq \nu^{-1}\int_{0}^{t/2}\|\nabla_{x',y'}\mathbb{G}(x-\cdot,y,t-s;\cdot)\|_{L^{\frac{4}{3}}}\|\left( u\omega\right) (s)\|_{L^{\frac{4}{3}}}ds\\
		&	\leq C\nu^{-1}\int_{0}^{t/2}(t-s)^{-\frac{1}{4}-\frac{1}{2}}\left(1+\frac{1}{12} \nu^{-2} t^2\right)^{-\frac{1}{8}}\|\left( u\omega\right) (s)\|_{L^{\frac{4}{3}}}ds\\
		&	\leq C\nu^{-3/4}\int_{0}^{t/2}(t-s)^{-1}\| \omega(s)\|_{L^2}\|u(s)\|_{L^{4}}ds\\
		&	\leq C\nu^{-3/4}t^{-1}\int_{0}^{t/2}\| \omega(s)\|_{L^2}\|u(s)\|_{L^{2}}^{\frac{3}{4}}\|\nabla^2u(s)\|_{L^{2}}^{\frac{1}{4}}ds\\
		&	\leq C\delta\epsilon^2\nu^{-3/4}t^{-1}\int_{0}^{t/2}(1+s)^{-1}ds\\
		&	\leq C\delta\epsilon^2\nu^{-3/4}.\\
	\end{aligned}
\end{equation}
The combination of equations \eqref{eq:36} and \eqref{eq:37} yields
 \begin{equation}\label{eq:38}
	\begin{aligned}
		\|I_{23}\|_{L^2}\leq C\delta\epsilon^2\nu^{-3/4}(1+t)^{-1}.\\
			\end{aligned}
	\end{equation}
If  we choose 
$$p_1=\frac{10}{9}, \quad p_2=\frac{5}{3}, \quad\theta=\frac{1}{2},$$
and $\delta$  is suitably selected and $c_0$ is sufficiently small, meaning
$$C_3\leq\frac{\delta}{2}, \quad \epsilon\leq c_0\nu^{3/4}.$$
Combining the \eqref{eq:29}, \eqref{eq:30}, \eqref{eq:33}, \eqref{eq:35} and \eqref{eq:38}, we have 
\begin{equation}
	\begin{aligned}
		\|\omega(t)\|_{L^2}
		&\leq \|I_1||_{L^2}+\|I_2||_{L^2}\leq\|I_1||_{L^2}+\|I_{21}||_{L^2}+\|I_{22}||_{L^2}+\|I_{23}||_{L^2}\\
		&\leq C_3(1+t)^{-1}\epsilon+C\delta(1+t)^{-1}\epsilon^2\nu^{-1/p_1+2\theta/p_1-3\theta/2}\\
		&+C\delta(1+t)^{-1}\epsilon^2\nu^{-1/p_2+2\theta/p_2-3\theta/2}+C\delta\epsilon^2\nu^{-3/4}(1+t)^{-1}\\
		&\leq \frac{\delta}{2}(1+t)^{-1}\epsilon.
	\end{aligned}
\end{equation}
This completes the proof of the Proposition \ref{pro4.2}.
\end{proof}
Now we give the proof of Theorem \ref{thm1}.

According to the Proposition \ref{pro4.1}, Proposition \ref{pro4.2}, it can be inferred that if 	$$\mathbf{H}(t):||\omega(t)||_{L^2}\leq \delta(1+t)^{-1}\epsilon,$$ and the initial data satisfies \eqref{eq:16}, then we have
	$$\mathbf{C}(t):	||\omega(t)||_{L^2}\leq\frac{\delta}{2} (1+t)^{-1}\epsilon$$
for	$t\in[0,\infty).$
	By applying Lemma \ref{lem2.1}, we obtain the global stability estimates.

\textbf{Conflict of Interest}\quad The authors declared that they have no conflict of interest.
\section{Acknowledgment}
 The authors thank to Prof. Shijin Deng and Binbin Shi  for suggesting this problem and many valuable discussions. The authors are supported by National Nature Science Foundation of China 12271357 and 12161141004 and Shanghai Science and Technology Innovation Action
Plan No. 21JC1403600.

\end{document}